\documentclass[12pt,reqno]{amsart}
\usepackage{amsmath, amscd}
\usepackage{amsfonts}
\usepackage{version}
\usepackage{amssymb}
\usepackage{amssymb,amsmath}
\usepackage{ifthen}
\usepackage{epsfig}
\usepackage{graphicx}
\usepackage{caption}
\usepackage{subcaption}

\newcommand{\R}{\mathbb R}

\newcommand{\al}{\alpha}

\newcommand{\C}{\mathbb C}

\newcommand{\G}{\mathcal{G}}           \newcommand{\RA}{\mathcal{R}}
\newcommand{\So}{\mathcal{S}}
       
          \newcommand{\Pp}{\mathcal{P}}
\newcommand{\D}{\mathbb D}

\newcommand{\Ao}{\mathcal A}

\newcommand{\Ff}{\mathfrak F}

\newcommand {\Hol}{\mathop{\rm Hol}\nolimits(\D,\C)}

\newcommand  {\Id} {\mathop{\rm Id}\nolimits}
\newcommand {\HolD}{\mathop{\rm Hol}\nolimits(\D)}

\renewcommand{\Re}{\mathop{\rm Re}\nolimits}

\newtheorem{theorem}{Theorem}[section]
\newtheorem{lemma}[theorem]{Lemma}
\newtheorem{proposition}[theorem]{Proposition}
\newtheorem{corollary}[theorem]{Corollary}
\newtheorem{question}{Question}
\newtheorem{conjecture}{Conjecture}

\newtheorem{definition}[theorem]{Definition}

\newtheorem{example}[theorem]{Example}

\numberwithin{equation}{section}

\numberwithin{equation}{section}

\begin{document}
\title{On one filtration of holomorphic functions}


\author[F. Jacobzon]{Fiana Jacobzon}
\address{F. Jacobzon: Department of Mathematics, Braude College of Engeneering, Karmiel 21982, Israel.}
\email{fiana@braude.ac.il}

\date{\today}

\begin{abstract}
In this work we consider a family of function classes constructed by means of the Gauss hypergeometric function
$
_2F_1(1,1;2;z) =-\frac{\log(1-z)}{z}.
$

We demonstrate that this family, in fact, constitutes classes of analytic functions subject to prescribed constraints on their derivatives. 
For these classes we obtain some geometric characteristics, including sharp coefficient estimates. 

Moreover, we show that this family naturally provides a filtration of infinitesimal generators, and investigate the corresponding dynamical behavior of the associated semigroups.
It is interesting that this filtration links to the Ma--Minda starlike functions.
\end{abstract}

\maketitle

\bigskip

\setcounter{equation}{0}
\section{Introduction and preliminary results}\label{sect-preli-geom}
\setcounter{equation}{0}

 Let $\D$ be the open  unit disk in the complex plane $\C$. Denote by $\Hol$ the set of holomorphic functions on $\D$, and by $\HolD := \Hol$, the set of all holomorphic self-mappings of $\D$.

 Let us denote by $\Ao$ the subset of $\Hol$ consisting of  functions normalized by $f(0)=f'(0)-1=0$, and by $\So$ the subset of $\Ao$ consisting of univalent functions.  
 
 Another important subclass of $\Hol$ is the Carath\'eodory class 
 \[
 \Pp:=\left\{ q\in\Hol:\ \Re q(z)>0,\quad  z\in\D \quad\mbox{and}\quad q(0)=1 \right\}.
 \]
 
Let us recall the famous Noshiro--Warschawski theorem (see, for example, \cite{Duren})
 \begin{theorem}\label{thm-NW}
 	If $f\in\Hol$ and $\Re f'(z)>0$, then $f$ is univalent in $\D$.
 \end{theorem}
 Thanks to this theorem, it begs to consider the class
 \[
 \RA:=\left\{ f\in\Ao:\ \Re f'\in\Pp  \right\},
 \]
 which is usually called the Noshiro--Warschawski class. More detailed information on this class can be found in \cite{Duren, E-V, GAW, G_K_book}.
 
 Recently, in several works \cite{BCDES, EJ-survey, ESS} it was noticed that every function  $f\in \RA$ is a semigroup generator (see Definition~\ref{def-generator} below). Moreover, elements of $\RA$ have very specified geometric and dynamical properties. These facts among  others served a base to develop theory of so-called filtrations of semigroup generators (see Definition~\ref{def-filt}).

It can be easily seen that if we replace the   Noshiro--Warschawski condition by a weaker one, namely, 
\begin{equation}\label{alpha-1}
f\in\Ao,\qquad \Re f'(z)>-\alpha
\end{equation}	
for some $\alpha>-1$, then there are non-univalent functions satisfying the last inequality. At the same time, we will see that for some positive $\alpha$ each function that satisfies \eqref{alpha-1} forced to be a generator.

\vspace{2mm}
Assume $f$ satisfies~\eqref{alpha-1}. In this paper the following questions are studied. 

\begin{question}
	What geometric properties does the function $f$ exibit?  
\end{question}

\begin{question}
	What analytical features characterize the function $f$?  
\end{question}

\begin{question}
	For which $\alpha$ does $f$ serve as a semigroup generator?
\end{question}

\begin{question}
	If $f$ is a semigroup generator, what are dynamical properties of this semigroup?
\end{question}


Note in passing, that the inequality \eqref{alpha-1} is closely connected 
the work \cite{EJ-lattice}, where  the set-theoretic properties of the two-parameter family of function classes defined by 
\[
\Re\left((s-1)\frac{f(z)}{z} +f'(z)\right) \ge st,\quad s>0, \, t\in[0,1),
\]
were studied.
Observe that \eqref{alpha-1} extends the last formula for $s=1$ to classes of functions whose derivatives may have negative real parts. In that work the Gau\ss \ hypergeometric function $_2F_1(1,s;s+1;z)$ plays a crucial role. Substituting $s=1$ we get
\[
_2F_1(1,1;2;z) =-\frac{\log(1-z)}{z}\,.
\]
This function will also prove to be important in the present study.

\vspace{2mm}
In what follows, we need an additional notion. Let $f, g\in\Hol$ and
\begin{equation*}\label{def-Omega}
\Omega=\{ \omega \in \HolD:\  \omega(0)=0 \}.
\end{equation*}
One says that $f$ is subordinate $g$ and writes $f\prec g$, if there exists a function $\omega\in\Omega$ such that $f(z) = g(\omega(z))$ for all $z\in\D.$ In the case where $g$ is univalent, conditions $f\prec g$ and $f(\D)\subseteq g(\D)$ are equivalent.

\vspace{2mm}

We will also make use of the main notions and facts from semigroup theory. For more details the reader can be refereed to \cite{B-C-DM-book, E-Sbook, SD}.
\begin{definition}\label{def-sg}
A family $\left\{ \phi _{t}\right\} _{t\geq 0}\subset\HolD$ is
called a {\sl one-parameter continuous semigroup} (or just {\sl semigroup}) if

(a) $\phi _{t+s}=\phi _{t}\circ \phi _{s},\ t,s\geq 0$; and

(b) $\lim\limits_{t\rightarrow 0^{+}}\phi _{t}=\Id$, where $\Id$ is
the identity map on $\D$ and the limit is taken with respect to the topology of uniform convergence on compact sets in $\D$.
\end{definition}

Moreover, according to the fundamental result by Berkson and Porta \cite{BE-PH} the family $\left\{ \phi _{t}\right\} _{t\geq 0}$ is differentiable with respect to its parameter
$t\geq 0,$ and the limit
\begin{equation}\label{generator}
f=\lim_{t\rightarrow 0^{+}}\frac{1}{t}\left( \Id-\phi _{t}\right),
\end{equation}
exists, and defines a holomorphic function on $\D$. Furthermore, $\phi _{t}$ is the solution of the Cauchy problem:
\begin{equation*}
\frac{\partial \phi _{t}(z)}{\partial t}+f(\phi _{t}(z))=0\quad \mbox{and} \quad \phi _{0}(z)= z\in \D.
\end{equation*}

\begin{definition}\label{def-generator}
The function $f\in\Hol$ defined by \eqref{generator} is called the  {\sl infinitesimal generator} of
the semigroup $\left\{ \phi _{t}\right\} _{t\geq 0}\subset\HolD.$
\end{definition}

Denote by $\G$ the class of all holomorphic generators on $\D$. 
 The following assertion contains a parametric representation of this class.
\begin{theorem}\label{th-gen-prop}
Let $f\in \Hol$. Then $f \in\G$ if and only if there are a point $\tau \in \overline{\D}$ and a function $q \in \Hol$ with $\Re q( z) \geq 0$, $z\in \D,$ such that
\begin{equation}\label{b-p}
f( z ) =\left( z-\tau \right) \left( 1-z\overline{\tau }\right) q( z),\quad z\in \D.
\end{equation}
Moreover, this representation is unique.
\end{theorem}
We notice that formula \eqref{b-p}  was obtained by Berkson and Porta in \cite{BE-PH} and is called the {\it Berkson--Porta representation}. 

We further focus on the case $\tau\in\D$.  Up to the M\"obius transformation $M_\tau\ \left(M_\tau(z)=\frac{\tau-z}{1-z\overline\tau}\right)$ of the unit disk, one can always consider generators such that $f(0) =0,$ or, what is the same, $\phi_t(0) =0$ for all $t\geq 0$. It follows from Theorem~\ref{th-gen-prop} that $f\in \Ao$ is an infinitesimal generator if and only if  $\displaystyle\Re \frac{f(z)}{z}> 0$ for all $z\in \D \setminus \{0\}$.
We denote
\begin{equation*}
\G_0:=\mathcal{A}\cap \G= \left\{ f\in\Hol: \ \frac{f(z)}z\in\Pp\right\}.
\end{equation*}

The class $\G_0$ is very important for the study of non-autonomous problems and geometric function theory (see, for example, \cite{B-C-DM-book, Duren, E-R-S-04, E-Sbook, R-S1}).

\vspace{2mm}
Over the years, the study of the asymptotic behavior of semigroups was mainly focused on the local/global rate of convergence and the growth estimates of a semigroup with respect to its parameter. Different estimates of the rate of convergence of semigroups were obtained, see, for example, the books \cite{B-C-DM-book, E-Sbook, SD}, the survey \cite{JLR} and references therein.

In view of subsequent results, we recall the following notion. 
\begin{definition}[see \cite{BCDES}]\label{def-squee}
Let $\{\phi _{t}(z)\}_{t\ge 0}$ be a semigroup. If there exists a constant $k> 0$ such that 
\begin{equation*}
|\phi _{t}(z)|\leq e^{-kt}|z|\quad \text{ for all } \quad z\in \mathbb{D}, 
\end{equation*}
then $\left\{ \phi _{t}\right\} _{t\geq 0}$  is said to be \textsl{exponentially squeezing semigroup with squeezing ratio $k$.}
\end{definition}
The following criterion for a semigroup to be exponentially squeezing holds.
\begin{theorem}\label{th-sqeezing}
Let $f \in \G_0$ and $\{\phi_{t}\}_{t\geq 0}$ be the semigroup generated by $f$. Then  $\{\phi_{t}\}_{t\geq 0}$ is exponentially squeezing with squeezing ratio $k>0$ if and only if $\displaystyle\Re\frac{f(z)}z\geq k$ on $\D$.
\end{theorem}

Another direction of research on semigroup properties is focuses on the possibility of analytic extension with respect to the semigroup parameter into a complex domain.
The problem is to identify conditions that allow such an extension, along with estimates of the sector in $\C$ to which this extension can be performed. For recent results in this direction see \cite{ACP, E-S-T, E-J}, see also \cite{E-R-S-19}.

More specifically, fix $\theta\in\left(0,\frac\pi2\right]$ and denote
\begin{equation}\label{sector}
\Lambda(\theta) = \left\{\zeta\in\C: |\arg \zeta| <\theta
 \right\}.
\end{equation}

\begin{theorem}  [Theorem~2.12 in \cite{E-S-T}]\label{thm-rotation}
Let $\{ \phi_t \}_{t \geq 0}$ be a semigroup of holomorphic self-mappings of $\D$ generated by $f\in\G_0$, and let $\alpha\in\left[0,1\right)$. Then $\{ \phi_t \}_{t \geq 0}$ extends analytically to the sector $\Lambda\left(\displaystyle\frac{\pi(1-\al)}2\right)$ in $\C$ if and only if $\displaystyle\left|\arg \frac{f(z)}z \right|\leq \frac{\pi\alpha}2$ on $\D\setminus \{0\}$.
\end{theorem}
In the same work \cite{E-S-T} a criterion for analytic extension was established.

\begin{theorem}\label{thm-analytic}
Let $f\in\G$ and  $\theta\in\left[0,1\right)$. 
Then the semigroup generated by $f$  extends analytically to the sector
$\quad \Lambda\left(\frac{\pi(1-\theta)}2\right)$ in $\C$ if and only if  
$$\left|\arg \frac{f(z)}z \right|\leq \frac{\pi\theta}2 \ \text{ for all } \ z\in \D\setminus \{0\}.$$ 
\end{theorem}

These directions in study of dynamical systems highlight the importance of classifying generators according to the dynamical properties of the semigroups they generate. To this end, we introduce a filtration (or parametric embedding) of infinitesimal generators.

\begin{definition}[see \cite{BCDES, EJ-survey, ESS}]\label{def-filt}
Let $J$ be a connected subset of $\R$.
A {\sl filtration} of $\G_0$ is a family
$\mathfrak{F}= \left \{\Ff_s\right\}_{s\in J},\
\mathfrak{F}_s\subseteq\G_0,$ such that $\mathfrak F_s\subseteq \mathfrak F_t$ whenever $s,t \in J$ and $s\le t$.\\
Moreover, if $\mathfrak F_s\subsetneq \mathfrak F_t$ for
$s<t\ s,t\in J$, then we say that the filtration $\left \{ \mathfrak F_s\right\}_{s\in J}$ is {\sl strict}.  In this case for $t\in J$ we set
\begin{equation}\label{boundary}
\partial\mathfrak{F}_t=\mathfrak{F}_t \setminus\bigcup_{s\in J, s<t}\mathfrak{F}_s.
\end{equation}
\end{definition}

Let a filtration $\mathfrak{F}= \left\{\Ff_s\right\}_{s\in J}$ be given. We now split Question 4 into two detailed parts for clarity.
\begin{itemize}
  \item   Determine the sharp squeezing ratio for all semigroups generated by elements of each filtration set $\Ff_s$;

  \item   Identify the maximal sector into which all these semigroups admit an analytic extension.
\end{itemize}

Note that if the boundaries defined by \eqref{boundary} are not empty for every $t\in J$, then the filtration $\mathfrak{F}$ is strict although for a strict filtration boundaries of its elements might be empty.

The study of  different classes of functions involves the searching for extremals. We will use the following notion:

\begin{definition}[see \cite{BCDES}]\label{def-extrem}
Let $\mathcal F\subset\mathcal{A}$. We say that a function $f_*\in\mathcal{F}$
is {\sl totally extremal} for $\mathcal F$ if for every $\lambda\in\C$ and  $r\in [0,1]$
\[
\min_{|z|=r} \Re \left(\lambda\frac {f(z)}{z}\right)\ge
\min_{|z|=r}\Re\left(\lambda \frac {f_*(z)}{z}\right)  \text{ for all } f\in\mathcal F.
\]
\end{definition}

\begin{definition}\label{def-extrem1}
We say that a filtration $\mathfrak{F}=\{\mathfrak F_s\}_{s\in J}$ admits a net $\left\{f_s\right\}_{s\in J}$ of {\sl
totally extremal functions} if for every $s\in J $, the function $f_s$ is totally extremal for the class~$\mathfrak F_s$.
\end{definition}
\bigskip

\section{The class $\mathcal{L_\alpha}$}
\setcounter{equation}{0}

Let us consider classes  of normalized holomorphic functions.
\begin{definition}
For $\alpha \geq-1$ define
\begin{equation}\label{L-by-estim-der}
\mathcal{L}_\alpha:=\{f\in\mathcal{A}:\Re f'(z)\geq-\alpha\}. 
\end{equation}

\end{definition}

Obviously, if $\alpha<\beta$ then $\mathcal{L}_\alpha \subsetneq\mathcal{L}_\beta$.

\begin{example}\label{ex-varphi-in-L}
  Fix $\alpha<-1$. The function $f(z)=z\varphi_\alpha (z)$ where 
  \begin{equation}\label{varphi}
\displaystyle\varphi_\alpha(z)=-(1+2\alpha)-2(1+\alpha)\frac{\log(1-z)}{z}, \quad z \in \D
  \end{equation}
belongs to the class $\mathcal{L}_\alpha $.\\
Indeed, a straightforward computation shows that $\displaystyle\Re f'(z)=-(1+2\alpha)+2(1+\alpha)\Re\frac{1}{1-z}$
 and $\displaystyle\inf_{z\in\D}\Re\frac{1}{1-z}=\frac{1}{2}$.  Hence, $\displaystyle\Re f'(z)>-\alpha$ as required.
\end{example}

\bigskip

\subsection{Analytic and geometric features of $\mathcal{L_\alpha}$}
\setcounter{equation}{0}

The next result establishes a representation of the elements in $\mathcal{L}_\alpha$.
\begin{lemma}\label{LM-elementL}
\begin{equation}\label{L-by-int-varphi}
  \mathcal{L}_\alpha =\left\{f\in \Ao:\,  f(z)=z\int_{\partial\D}\varphi_\alpha(z\overline{\zeta}) d\mu(\zeta)\right\},
\end{equation}
where $\varphi_\alpha$ defined by~\eqref{varphi} and $\mu$ is a probability measure on the unit circle.
\end{lemma}
\begin{proof}
Assume that $\Re f'(z)> -\alpha$. Denote $\displaystyle p(z):=\frac{f(z)}{z}$ on $\D \setminus \{0\}$. Thus, $ f'(z)=p(z)+zp'(z)$ and
\begin{equation*}\label{calculations-Re-p}
  \frac{f'(z)+\alpha}{1+\alpha} = \left(z\cdot \frac{p(z)+\alpha}{1+\alpha}\right)' .
\end{equation*}
Since $\displaystyle\Re \frac{f'(z)+\alpha}{1+\alpha} >0$, then by Herglotz-Riesz representation theorem for holomorphic functions we have 
\begin{equation*}\label{calculations-Re-p1}
  \frac{f'(z)+\alpha}{1+\alpha}= \int_{\partial \D} \frac{1+z\overline{\zeta}}{1-z\overline{\zeta}} d\mu(\zeta).
 \end{equation*}
Therefore,
\begin{equation*}\label{calculations-Re-p2}
 z\cdot \frac{p(z)+\alpha}{1+\alpha}= \int_{\partial \D} \left(\frac{-2}{\overline{\zeta}} \log (1-z\overline{\zeta})-z\right) d\mu(\zeta)
 \end{equation*}
and 
\begin{equation*}\label{calculations-Re-p2}
 p(z)= \int_{\partial \D} \varphi_\alpha(z \overline{\zeta})d\mu(\zeta)
 \end{equation*} 
Thus $\displaystyle \frac{f(z)}{z}=\int_{\partial\D}\varphi_\alpha(z\overline{\zeta}) d\mu(\zeta)$.\\
The arguments above work in both directions. That is, 
classes defined by formulae~\eqref{L-by-int-varphi}  and~\eqref{L-by-estim-der} are coincide.  This completes the proof.
\end{proof}

\begin{lemma}\label{lem-phi-properties}
The functions $\varphi_\alpha$  have the following properties:
\begin{itemize}
  \item [(i)] Taylor's coefficients of $\varphi_\alpha(z)=\sum_{n=0}^{\infty} \rho_n z^n$ are $\rho_0=1$ and $\displaystyle\rho_n=\frac{2(1+\alpha)}{n+1}$, $n\geq 1$;
      
  \item [(ii)] The Fekete--Szeg\"o functional is $|\rho_3-\lambda \rho_2 ^2|=\displaystyle\frac{1+\alpha}{18}\cdot|9-8\lambda|$, $\lambda \in \C$;
      
  \item [(iii)] The function $\varphi_\alpha$, $\alpha > -1$, is a univalent convex function;
      
  \item[(iv)] Let $\, -1<\alpha<\beta$,  then $\varphi_\alpha \prec \varphi_\beta$.
\end{itemize}  
\end{lemma}

\begin{proof}
One verifies (i) and (ii) by direct calculations. 
To see that the property (iii) holds, denote  $C(z):=\displaystyle\frac{\log(1-z)}{z}$ on $\D\setminus \{0\}$ with $C(0)=-1$. It is enough to show that $C$ is a univalent convex function.   
Note that  $c(z)=\displaystyle\frac{1}{z-1}$ maps conformally the unit disk into a half-plane, so it is a univalent convex function. By Theorem 2 in \cite{Libera} the function $C(z)$ also is univalent convex function. This proves (iii).

Next we need to show that there exists a function $w\in \Omega$  such that
\begin{equation*}\label{aaa}
-(1+2\alpha)-2(1+\alpha)\frac{\log(1-z)}{z}=-(1+2\beta)-2(1+\beta)\frac{\log(1-w(z))}{w(z)}, 
\end{equation*}
or, equivalently
\begin{equation*}\label{aab}
\frac{\beta-\alpha}{1+\beta}- \frac{1+\alpha}{1+\beta}\cdot \frac{\log(1-z)}{z}=-\frac{\log(1-w(z))}{w(z)}. 
\end{equation*}
Obviously, $\displaystyle\min_{w\in\Omega} \Re \left( -\frac{\log(1-w(z))}{w(z)}  \right)=\log 2$, $z \in \D$. 
Denote 
$$\displaystyle\sigma(z):=\frac{\beta-\alpha}{1+\beta}- \frac{1+\alpha}{1+\beta}\cdot \frac{\log(1-z)}{z}.$$
We need to verify that $\displaystyle\sigma(z) \prec -\frac{\log(1-w(z))}{w(z)}$ or equivalently, that 
 $\sigma(-1)>\log 2$.  This is true, since   $\log 2<1$. The proof is complete.
\end{proof}

Lemma~\ref{LM-elementL}, together with assertion (iii) of Lemma~\ref{lem-phi-properties}, implies
\begin{corollary}
 If $f\in\mathcal{L}_\alpha$ then $\displaystyle\frac{f(z)}{z}\prec \varphi_\alpha (z)$. 
\end{corollary}

Estimates on coefficient operators remain a topic of ongoing interest in the geometric function theory community. 
\begin{proposition}\label{propo-FS-coeff-for-L}
Let $f  \in \mathcal{L}_\alpha$ and $1+\sum_{n=1}^{\infty} \phi_n z^n$ is its Taylor series. Then 
\begin{itemize}
  \item [(i)] $\displaystyle|\phi_n|\leq \frac{2(1+\alpha)}{n}$, $n\geq 1$;
  \item [(ii)] $|\phi_3-\lambda\phi_2^2|\leq\displaystyle \frac{2(1+\alpha)}{3}\max\{1,|6\lambda-1|\}$ for all $\lambda \in \C$ and this inequality is sharp. 
\end{itemize}
\end{proposition}
\begin{proof}
Using  assertion(ii) of Lemma~\ref{lem-phi-properties} one gets
\begin{equation*}
\frac{f(z)}{z}=\int_{\partial \D}\left(1+ \sum_{n=1}^{\infty}\frac{2(1+\alpha)}{n+1}z^n \overline{\zeta}^n \right) d\mu(\zeta).
\end{equation*}
So,  
\begin{equation*}
f(z)=z+   \sum_{n=1}^{\infty} \left(\frac{2(1+\alpha)}{n+1}\int_{\partial \D} \overline{\zeta}^n  d\mu(\zeta)\right) z^{n+1}.
\end{equation*}
Therefore,  $\displaystyle |\phi_n|\leq \frac{2(1+\alpha)}{n}\left|\int_{\partial \D} \overline{\zeta}^{n-1}  d\mu(\zeta)\right|\leq \frac{2(1+\alpha)}{n}$,  $n\geq 1$. This proves (i).\\
Next one calculates 
\begin{equation*}
\phi_3-\lambda\phi_2^2=\displaystyle\frac{2(1+\alpha)}{3}\left(\int_{\partial \D}\overline{\zeta}^2 d\mu(\zeta)-\frac{3\lambda}{2}\left(\int_{\partial \D}\overline{\zeta} d\mu(\zeta)\right)^2\right).  
\end{equation*}
Define a function 
\begin{equation*}
\widetilde{f}(z):=2\sum_{n=1}^{\infty} \left(\int_{\partial \D} \overline{\zeta}^{n-1}  d\mu(\zeta)\right) z^n.
\end{equation*}
Then 
\begin{equation*}
\widetilde{f}(z)=2z \int_{\partial \D} \frac{d\mu(\zeta)}{1-z\overline{\zeta}^n}=z\left( 1+ \int_{\partial \D} \frac{1+z\overline{\zeta}}{1-z\overline{\zeta}}d\mu(\zeta) \right).
\end{equation*} 
Denote $p(z):=\displaystyle\int_{\partial \D} \frac{1+z\overline{\zeta}}{1-z\overline{\zeta}}d\mu(\zeta)$, then  
$p(z)\in\mathcal{P}$ and $ \widetilde{f}(z)=z+zp(z)$.  Let the Taylor extension of $p$ be $\sum_{n=0}^{\infty}p_nz^n$, then
\begin{equation*}
|\phi_3-\lambda\phi_2^2|=\frac{2(1+\alpha)}{3} \cdot \left|p_2 - 3\lambda p_1^2\right| \leq \frac{2(1+\alpha)}{3} \cdot\max\{1,|6\lambda -1|\}
\end{equation*}
The last inequality is sharp by Theorem 1. in \cite{Eframidis}.
\end{proof}
To study further geometric properties of the class $\mathcal{L}_\alpha$, let us denote 

\vspace{1mm}
$\ K(\mathcal{L}_\alpha):= \displaystyle \inf_{f\in\mathcal{L}_\alpha} \Re \frac{f(z)}{z}$,  \ \ $B(\mathcal{L}_\alpha):=\displaystyle\sup_{f\in\mathcal{L}_\alpha} \arg \frac{f(z)}{z}$ \ and 
$\quad \displaystyle  \alpha_*=\frac{2\log 2 -1}{2-2\log2}\approx 0.629$.
\begin{theorem}\label{Thm-K} Let $\alpha\geq-1$, then
\begin{itemize}
  \item [(i)]  $K(\mathcal{L}_\alpha)=-(1+2\alpha)+2(1+\alpha)\log(2)$;
  \item [(ii)]  The following are equivalent:  \vspace{1mm}
\begin{itemize}
  \item [(1)]$\alpha\leq\alpha_*$; \vspace{1mm}
  \item [(2)]The image of $\displaystyle\frac{f(z)}{z}$ does not contain the origin;\vspace{1mm}
  \item [(3)] $B(\mathcal{L}_\alpha)=\displaystyle\frac{2}{\pi} \cdot \max_{0<\theta<\pi} \arg \varphi _\alpha (e^{i\theta})$;
  \item [(4)] $\mathcal{L}_\alpha \subset \mathcal{G}$.
\end{itemize}
\end{itemize}
\end{theorem}
\begin{proof}
By assertion (iii) of Lemma~\ref{lem-phi-properties}, the function $\varphi_\alpha$ is convex. 
Consequently, for any probability measure $\mu$ on the unit circle the function $p(z)=\int_{\partial\D}\varphi_\alpha(z\overline{\zeta}) d\mu(\zeta)$ takes values in the image of $\varphi_\alpha$. Example~\ref{ex-varphi-in-L} shows that $\varphi_\alpha \in \mathcal{L}_\alpha$. Therefore,
\begin{equation*}
K(\mathcal{L}_\alpha)=\inf_{z\in\D}\Re \varphi(z)=\varphi_\alpha(-1)=-(1+2\alpha)+2(1+\alpha)\log(2),
\end{equation*}
which proves assertion (i).
The conditions $K(\mathcal{L}_\alpha)\geq 0$ and $\alpha\leq\alpha_*$ are equivalent, which in turn implies that $0 \notin Im \left(\displaystyle\frac{f(z)}{z}\right)$. The quantity $B(\mathcal{L}_\alpha)$ represents the angle between the two lines through the origin that are tangent to the boundary of $Im \left(\displaystyle\frac{f(z)}{z}\right)$. If $K(\mathcal{L}_\alpha)\geq 0$, then by the definition of the class $\mathcal{G}$, we have 
$f(z)\in \mathcal{G}$ , and the assertion (ii) follows. This completes the proof.
\end{proof}

\bigskip

\subsection{Dynamical properties of $\mathcal{L_\alpha}$}
\setcounter{equation}{0}

The  quantities  $B(\mathcal{L}_\alpha)$  and $K(\mathcal{L}_\alpha)$ not only reflect the geometric features of the classes $\mathcal{L}_\alpha$, $\alpha\leq \alpha_*$, but also determine significant dynamical characteristics of semigroups generated by their elements. 

Let $\G_1\subset \G$. If $K(\G_1)>0$, then by Theorem~\ref{th-sqeezing} semigroups generated by elements of $\G_1$ have exponential rate $\exp(-tK(\G_1))$ of convergence, uniformly on the whole unit disk and this rate is sharp. 

\begin{corollary}
  Let $-1\leq\alpha\leq\alpha_*=\displaystyle\frac{2\log 2 -1}{2-2\log2}$, then each element of 
 $\mathcal{L}_\alpha$ is generator \\  \vspace{1mm} of exponentially squeezing semigroup with sharp squeezing ratio \\  \vspace{1mm}
  $K(\mathcal{L}_\alpha)=-(1+2\alpha)+2(1+\alpha)\log(2).  $
\end{corollary}

If $B(\G_1)<1$, then by Theorem~\ref{thm-analytic} semigroups generated by elements of $\G_1$ admit analytic extension to the sector $\Lambda\left(\frac\pi2 (1-B(\G_1) )\right)$ and the angle of opening of this sector is the maximal one.

\begin{corollary}
  Let $-1\leq\alpha\leq \alpha_*$ and $B(\mathcal{L}_\alpha)=\displaystyle\frac{2}{\pi} \cdot \max_{0<\theta<\pi} \arg \varphi _\alpha (e^{i\theta}).$  \\
\vspace{1mm}
Then semigroups generated by any $f\in\mathcal{L}_\alpha$ admit analytic extension to the sector\\ $\Lambda\left(\frac\pi2 (1-B(\mathcal{L}_\alpha) )\right)$ and this sector is the maximal one.
\end{corollary}

\bigskip

\subsection{The family $\mathcal{L}$}
\setcounter{equation}{0}

Consider the family $\mathcal{L}=\{\mathcal{L}_\alpha\}_{\alpha\geq-1}$. 
The following result characterizes this family from a set-theoretic perspective.
\begin{theorem}\label{th-F-alpha}
The following assertions hold: 
  \begin{itemize}
\item[(a)] The family $\mathcal{L} =\left\{ \mathcal{L}_\alpha\right\}_{\alpha \in[-1,\alpha_*]}$ forms a filtration with 
      $\mathcal{L}_{-1}=\{\Id\}$,   $\mathcal{L}_0=\mathcal{R}$  and  
    ${\mathcal{L}_{\alpha_*}=\displaystyle\left\{f\in \mathcal{A}:\,  \frac{f(z)}{z}=\frac{-\!\log 2+\!\int_{\partial\D^2}  \!\! F_1(1,1;2;z\overline{\zeta}) d\mu(\zeta)}{1-\log 2}\right\}} $.  
   \item[(b)] The filtration $\mathcal{L}$ admits a net $\left\{f_\alpha\right\}_{\alpha\in[-1,\alpha_*]}$ of totally extremal functions  $f_\alpha(z) :=z\varphi_\alpha(z)$.      
 \item[(c)] For every $\alpha\in[-1,\alpha_*]$ we have $f_\alpha\in\partial\mathcal{L}_\alpha:=\displaystyle\mathcal{L}_\alpha \setminus \bigcup_{-1\leq s<\alpha}\mathcal{L}_s$.  \\
 Consequently, the filtration $\mathcal{L}$ is strict. 
  \end{itemize}
\end{theorem}

\bigskip

\section{Open questions}
\setcounter{equation}{0}
\setcounter{question}{0}

\begin{question}
  Describe the boundaries $\displaystyle \partial\mathcal{L}_\alpha=\mathcal{L}_\alpha \setminus \bigcup_{-1\leq s<\alpha}\mathcal{L}_s$. \
\end{question}

\begin{conjecture}
The boundary $\partial\mathcal{L}_\alpha$ coincides with the set $\left\{e^{-i\theta}f_{\alpha}(e^{i\theta}z): \theta\in\R  \right\}$ for every $-1\leq\alpha\le \alpha_*$.
\end{conjecture}

The criterion for the analytic extension in sector  $\Lambda\left(\frac\pi2 (1-B(\mathcal{L}_\alpha) )\right)$  is $B(\mathcal{L}_\alpha)<1$.  
We have
$B(\mathcal{L}_\alpha)=\displaystyle\frac{2}{\pi} \cdot \max_{0<\theta<\pi} \arg \varphi _\alpha (e^{i\theta}).$
\vspace{-2mm}
\begin{question}
 Find explicitly or estimate the function $B(\mathcal{L}_\alpha)$ for $-1\leq\alpha \leq\alpha_*$. 
\end{question}

Recall that every function $f\in  \mathcal{L}_\alpha $  admits the representation  $$f(z)=z\int_{\partial\D}\varphi_\alpha(z\overline{\zeta}) d\mu(\zeta),$$ 
where $\displaystyle\varphi_\alpha(z)=-(1+2\alpha)+2(1+\alpha)_2F_1(1,1;2;z), \  z\in \D$\\ is univalent convex function for $\alpha>-1$.  Hence, $\displaystyle\frac{f(z)}{z}\prec \varphi_\alpha(z)$.

Thus, there exists $\omega \in \Hol(\D)$ with $\omega(0)=0$ such that  $u(z):=\displaystyle\frac{f(z)}{z}=\varphi_\alpha\left((\omega(z)\right)$, $u(0)=1$  
$\quad\Rightarrow$ $\quad\omega(z)=\varphi_\alpha^{-1}\circ u(z)$. 
\begin{question}
Determine conditions under which a function  $\omega\in \Hol(\D)$ with ${\omega(0)=0}$ can be represented as $\ \omega(z)=\varphi_\alpha^{-1}\left( \int_{\partial\D}\varphi_\alpha(z\overline{\zeta}) d\mu(\zeta)\right)$ for some $-1<\alpha \leq \alpha_*$.
\end{question}

\vspace{2mm}

\end{document}